\documentclass[11pt,a4paper,leqno]{amsart}
\usepackage{amsmath,amsfonts,amssymb,amsthm}
\usepackage[english]{babel}
\usepackage{datetime}
\usepackage{enumerate}
\usepackage[utf8]{inputenc}  
\usepackage{hyperref}
\usepackage{xcolor}

\theoremstyle{plain}
\newtheorem{theo}{Theorem}[section]
\newtheorem{lem}[theo]{Lemma}
\newtheorem{propo}[theo]{Proposition}

\theoremstyle{definition}

\newtheorem{remark}[theo]{Remark}

\def\finremark{	\hfill $\boxtimes$}
\numberwithin{equation}{section}
\newtheorem{propoletra}{Proposition}

\newtheorem{thmletra}{Theorem}

\def\D{\mathbb{D}}

\def\E{\mathbf{E}}
\def\V{\mathbf{V}}
\def\P{\mathbf{P}}

\def\K{\mathcal{K}}

\title{Galton-Watson processes, simple varieties of trees and Khinchin families}
\author[V.\,J. Maci\'a]{V\'{\i}ctor J. Maci\'a}
\address[V\'{\i}ctor J. Maci\'a]{Departamento de Matem\'aticas, CUNEF Universidad, Madrid, Spain.}
\email{victor.macia@cunef.edu}

\subjclass{30B10, 60J80, 05C05}
\keywords{Power series with non-negative coefficients, power series distributions, Khinchin families, Lagrange inversion, Galton-Watson processes, random trees, simple varieties of trees, extinction probability.}

\begin{document}

	\begin{abstract}

		In this note, we introduce a unified analytic framework that connects Bienaymé–Galton–Watson processes, simple varieties of trees and Khinchin families. Using Lagrange’s inversion formula, we derive new coefficient-based expressions for extinction probabilities and reinterpret them as boundary phenomena tied to the domain of the inverse of the solution to Lagrange’s equation. This perspective reveals an additional link between combinatorial and probabilistic models, simplifying classical arguments and yielding new results. It also leads to a computationally efficient method for simulating probabilities associated to Galton–Watson processes via power series coefficients. 
	\end{abstract}
	\maketitle
	
	\vspace{-.65 cm}
	\section{Introduction}

	\begingroup
\enlargethispage{-2.5\baselineskip}
		Power series distributions play a significant role in the study of rooted trees (see, for instance, \cite{Drmota_paper,Drmota, Flajolet, Janson2012,Pitman,Pitman2}). Although their primary role has traditionally been that of a tool rather than an object of intrinsic interest, they provide a natural bridge between combinatorial and probabilistic perspectives, offering valuable insights into the theory of random trees. In this note, using Khinchin families and Lagrange's inversion, we relate the probabilistic behaviour of Galton–Watson processes to the analytic behaviour of the generating function of their combinatorial counterpart: the simple varieties of trees. 
		\medskip
		
			A Bienaymé-Galton–Watson (henceforth Galton–Watson) process is a branching model where every individual independently produces a random number of children following a given offspring distribution. These processes can be traced back to Bienaymé’s work on extinction probabilities in the 1840s and were further developed by F.~Galton and H.~W.~Watson in 1874--1875, see \cite{Bienayme1845, Galton1873, GaltonWatson1875} and \cite{Watson1873}. For a thorough modern treatment of branching processes, see \cite{AthreyaNey}; for further historical background, see \cite{ Guttorp1992,Harris1963}.
			\medskip
			
		 Their motivating question can be summarized as: \textit{Will a given family line eventually die out?}. It is a classical result that the extinction probability $q$ of a Galton-Watson process is the smallest non-negative fixed point of the offspring’s probability‐generating function $G_Y(z)$, namely
		\[
		G_Y(q)
		= \sum_{n=0}^\infty \P(Y=n)\,q^n
		= q.
		\]
		
		The first complete proof of this result was given by J.~F.~Steffensen in his 1930 note \cite{Steffensen1930} (in Danish) and its 1933 expansion \cite{Steffensen1933}, see also the unpublished proof by C.\,M. Christensen in \cite{Guttorp1992}. For additional historical details and an English translation of Steffensen’s 1930 note, see Guttorp’s report \cite{Guttorp1992}.
		  \clearpage
	\endgroup
		
	\newpage
To each power series  $f \in \K$ (see the definition below) with radius of convergence $R>0$ we can associate a family of random variables $(X_t)_{t \in [0,R)}$, the Khinchin family associated to $f$. For each $t \in (0,R)$ the random variable $X_t$ has mass function
	\begin{align*}
	\P(X_t = n) = \frac{a_n t^n}{f(t)}\,, \quad \text{ for any } n \geq 0\,,
	\end{align*}
	with $X_0 \equiv 0$. The theory of Khinchin families examines the behaviour of this family of random variables, and its normalized version, as the parameter $t \uparrow R$. See, for instance, \cite{CFFM1, CFFM2, CFFM3, K_dos}, and more recently \cite{MaciaThesis, MaciaGaussian}.
	\smallskip


One of the main results of this note is the following new coefficient-based expression for
the probability of extinction of a Galton–Watson process: let $\psi \in \K$ be a power series with radius of convergence $R_{\psi}>0$ and Khinchin family $(Y_t)$ and let $Y_t$ be the offspring distribution of our process, then the probability of extinction $q(t)$ is given by:
	\begin{align}\label{eq: extinction_intro}
	q(t) = \sum_{n = 1}^{\infty}\frac{A_n t^{n-1}}{\psi(t)^n} = \frac{g(t/\psi(t))}{t}\,, \quad \text{ for any } t \in (0,R_{\psi})\,,
	\end{align}
 here $A_n$ is the $n$-th coefficient of the solution of Lagrange's equation $g(z) = z\psi(g(z))$ with data $\psi \in \K$. These coefficients can be computed in terms of $\psi$ by using Lagrange's inversion formula. We refer the reader to Theorem~\ref{thm:extinction_tree} below for a detailed statement and proof of this result.
\smallskip

A simple variety of trees is a family of weighted rooted plane trees in which each node receives a weight that depends on the number of its children (outdegree). The weight of the whole tree is the product of the different node-weights, see \cite[p. 452]{Flajolet}. Throughout this note we choose those node-weights to be the coefficients of a power series $\psi\in\mathcal{K}$ (see eq.~\eqref{eq:weights} below). For such a fixed $\psi$, the sequence $A_n$ records the sum of the weights over all rooted plane trees with $n$ nodes (see Proposition \ref{propo: weights} below). Specific choices of $\psi$ yield concrete families of combinatorial trees; for instance, $\psi(z)=e^{z}$ corresponds to rooted labelled Cayley trees, while $\psi(z)=1/(1-z)$ corresponds to rooted plane trees.
\smallskip 

	It turns out that the radius of convergence of the solution to Lagrange’s equation with data \(\psi\) is intimately related to the extinction probability of the corresponding Galton–Watson process with offspring distribution \(Y_t\), marking the critical parameter at which the process shifts (phase transition) from almost sure extinction to positive survival probability. 
\medskip

	For instance if we fix $\psi(z) = e^z$ and denote $(Y_t)$ the Khinchin family associated to $\psi$, then, for each $t\geq 0$, the random variable $Y_t$ is a Poisson random variable of parameter $t \geq 0$, and therefore the probability of extinction of the Galton-Watson process with offspring distribution $Y_t$ is given by the closed formula
	\begin{align*}
	q(t) = \sum_{n = 1}^{\infty}\frac{n^{n-1}t^{n-1}}{n!e^{nt}} = \frac{g(t/\psi(t))}{t}\,, \quad \text{ for any } t > 0\,.
	\end{align*}
Here, we simultaneously consider all possible Galton–Watson processes with offspring distribution given by a Poisson random variable with parameter \( t > 0 \). Notice that $q(t) \equiv 1$, for any $t \leq \tau = 1$ (to wit, the series above is actually equal to 1 for $t \leq 1$) and also that the radius of convergence of $g(z)$ is given by $\rho = \tau/\psi(\tau) = 1/e$, see Theorem \ref{thm: radius_g} below. When we cross the radius of convergence of $g(z)$, that is, when $t>\tau = 1$, the inversion relation \( {t}/{\psi(t)} = g^{-1}(t) \) no longer holds, and consequently the extinction probability should satisfy the inequality \( q(t) < 1 \). Observe that for $t> \tau = 1$, the function $t/\psi(t) = te^{-t}$ is strictly decreasing, see Lemma \ref{lemma: monotony_t/psi(t)} below.
	\smallskip
	
	After deriving formula \eqref{eq: extinction_intro}, we analyze the probability of extinction function in detail, establishing its continuity, differentiability (except at one point) and characterizing its asymptotic behavior as  the parameter \(t\uparrow R_{\psi}\). An explicit upper bound on the extinction probability function is also derived. Moreover, we establish several precise relationships between Khinchin families and random variables arising in Galton–Watson processes. We identify some of these random variables as appropriately reparameterized Khinchin families. Beyond its theoretical interest, this framework offers an efficient method for simulating probabilities associated to Galton-Watson processes by means of power series and the numerical computation of its coefficients. 
	We believe that this perspective will be valuable to both combinatorialists and probabilists, as it provides an explicit bridge between these viewpoints.

\medskip
	
This note aims to be self-contained and accessible to researchers with backgrounds in probability, combinatorics, or complex analysis. Section~2 introduces the fundamental concepts of Khinchin families. Section~3 reviews rooted plane trees and Galton–Watson processes. Section~4 recalls Lagrange’s inversion formula and gives a closed form expression for the radius of convergence of the power series that solves Lagrange's equation. Section~5 provides a combinatorial interpretation of the solution to Lagrange’s equation in terms of weighted rooted plane trees. Finally, Section~6 constitutes the core of this note: it develops the parametric Galton–Watson process, derives explicit coefficient-based formulas for various probabilities, including extinction probabilities, and explores their analytic and probabilistic behaviour.

	\subsection{Notations}
	
	For random variables \( X \) and \( Y \), the notation \( X \overset{d}{=} Y \) indicates that \( X \) and \( Y \) are equal in distribution, that is,
	\[
	\P(X \in B) =\P(Y \in B) \quad \text{for all Borel sets } B \subset \mathbb{R}.
	\]
	
	The unit disk in the complex plane \( \mathbb{C} \) is denoted by \( \mathbb{D} \). The open disk of center \( a \in \mathbb{C} \) and radius \( R > 0 \) is denoted \( \mathbb{D}(a, R) \), and its closure is denoted by \( \overline{\mathbb{D}(a, R)} \).
	\smallskip

	For any power series \( g(z) = \sum_{n=0}^\infty a_n z^n \in \mathcal{K} \) with radius of convergence \( R > 0 \), we define
	\[
	Q_g \triangleq \gcd\{n \geq 1 : a_n \neq 0\} = \lim_{N \to \infty} \gcd\{1 \leq n \leq N : a_n \neq 0\}.
	\]
	
	\subsection*{Acknowledgements}
The author thanks Professor José L. Fernández for many valuable conversations, and also Odi Soler i Gibert for comments on an early draft of this note.
	
	\section{Khinchin families}
	
	In this section, we introduce the key elements of the theory of Khinchin families that are necessary to develop this note. A more comprehensive treatment can be found in \cite{CFFM1}, \cite{CFFM2}, \cite{CFFM3} and \cite{K_dos} and more recently, in \cite{MaciaThesis} and \cite{MaciaGaussian}.
	\medskip
	
	We denote by $\K$ the class of non-constant power series
	$$f(z)=\sum_{n=0}^\infty a_n z^n$$  with positive radius of convergence $R>0$,  which have non-negative Taylor coefficients and such that  $a_0>0$. Since $f \in \K$ is non-constant, at least one coefficient other than $a_0$ is positive.
	\medskip
	
	The \textit{Khinchin  family} of  such a  power series $f \in \K$ with radius of convergence $R>0$ is the family of random variables $(X_t)_{t \in [0,R)}$ with values in $\{0, 1, \ldots\}$ and with mass functions given by
	$$
	\P(X_t=n)=\frac{a_n t^n}{f(t)}\, , \quad \mbox{for each $n \ge 0$ and $t \in (0,R)$}\, .$$  Notice that $f(t)>0$ for each $t \in[0,R)$. We define $X_0\equiv 0$. 
	Formally, the definition of $X_0$ is consistent with the general expression for $t \in (0,R)$, with the convention that $0^0 = 1$, meaning that $\P(X_0 = 0) = 1$.
	\subsection{Mean and variance functions.} \label{section: mean and variance}
	
	For the mean and variance of $X_t$, we reserve the notation
	$m_f(t)=\E(X_t)$ and $\sigma_f^2(t)=\V(X_t)$, for $t \in [0,R)$. In terms of the power series $f \in \K$, the  mean and the variance of $X_t$ may be  written as
	\begin{align}\label{eq:m and sigma in terms of f}
		m_f(t)=\frac{t f^\prime(t)}{f(t)} = t\frac{d}{dt}\ln(f(t)), \qquad \sigma_f^2(t)=t m_f^\prime(t)\, , \quad \mbox{for $t \in [0,R)$}\,.
	\end{align}
	Because the variance is always positive, for any $t \in (0,R)$, we have that $m_{f}$ is an increasing function. For further results about the range of the mean, see \cite{CFFM1}.
	\smallskip

	\subsection{The classes \(\mathcal K_s\) and \(\mathcal K^\star\)}
	
	A power series \(f(z)=\sum_{n\ge0}a_nz^n\in\mathcal K\) lies in \(\mathcal K^\star\) if 
	\[
	M_f:=\lim_{t\uparrow R}m_f(t)>1,
	\]
	so there is a unique \(\tau\in(0,R)\) (the \emph{apex}) with \(m_f(\tau)=1\), equivalently \(\tau f'(\tau)=f(\tau)\). 
	The existence of the apex prevents \( f \) from being a polynomial of degree 1. In particular, it implies that the second derivative at \( \tau \) satisfies \( f^{\prime\prime}(\tau) > 0 \). 
	\smallskip
	
	If \(g(z)=z^Nf(z)=\sum_{n\ge N}b_nz^n\), its associated Khinchin family \((Y_t)\) satisfies
	\[
	\P(Y_t=n)=\frac{b_nt^n}{g(t)},\quad \text{ for all } n\ge N,
	\]
	and \(Y_t\overset d= X_t+N\), where \(X_t\) is the Khinchin family of \(f\). The class \(\mathcal K_s\) comprises all such shifted power series (with associated shifted Khinchin family). See \cite{K_dos}, for further details.

	\section{Rooted trees}
	
	In this section, we introduce basic definitions associated with rooted trees. Additionally, we provide an overview of the basic definitions related to the Galton-Watson processes.
	
	\subsection{The combinatorial class of rooted plane trees}
	
	Rooted trees are graphs without cycles, with a distinguished node called the root. A \textit{rooted plane tree} is a rooted tree in which the children of each node are assigned a left-to-right order. Equivalently, it is a rooted tree embedded in the plane, where the relative order of siblings is preserved.  
	\smallskip
	
	Rooted plane trees can be defined recursively as follows:  
	\begin{itemize}
		\item A single node is a rooted plane tree.
		\item If \( T_1, T_2, \dots, T_k \) are rooted plane trees, then a new rooted plane tree can be formed by adding a root node and attaching \( T_1, T_2, \dots, T_k \) as its ordered children (from left to right).
	\end{itemize}
	
	We denote by \( (\mathcal{G}, |\cdot|) \) the combinatorial class of rooted plane trees, where the size function \( |\cdot| \) gives the number of nodes in a given tree. This family of rooted trees is codified, in the language of the combinatorial classes, by means of the specification
	\begin{align*}
		\mathcal{G} = \mathcal{Z} \times \text{SEQ}(\mathcal{G})\,.
	\end{align*}
	This structure assumes that the descendants of the root are ordered as a list of rooted trees. See \cite{Flajolet} [p. 65] for further details.
	
	\subsection{The Polish space of rooted plane trees} We denote by \(\mathbb{T}\) the set of rooted plane trees in which every node has a finite number of descendants (although the trees themselves may be infinite), see, for instance, \cite{Kesten} for further details. It is well known that $\mathbb{T}$, endowed with the metric
$$
	\delta(t,t') = 2^{-\sup\{h \in \mathbb{N} \cup \{0\}\, : \, r_h(t) = r_h(t')\}},
$$
	is a Polish space (see, for instance, \cite{Kesten}). Here, \(r_h(t)\) denotes the restriction function that gives the tree obtained by cutting \(t\) at level \(h\) (i.e., keeping the nodes at level \(h\) as leaves).
	\smallskip
	
	We can express a partition of this space as \( \mathbb{T} = \mathcal{G} \cup \mathcal{G}_{\infty} \). Here, \( \mathcal{G}_{\infty} = \mathbb{T} \setminus \mathcal{G} \) denotes the set of infinite rooted plane trees, where each node has a finite number of descendants.

	\subsection{Galton–Watson processes}
	
	A Galton–Watson process begins with a single ancestor whose offspring count \(Y\) follows the distribution \((p_k)_{k\ge0}\); each descendant independently reproduces according to the same law in successive generations, see, for instance,  \cite{AthreyaNey} and \cite{Neveu}. If \(Z_n\) denotes the population at generation \(n\), then
	\[
	Z_{n+1}=\sum_{j=1}^{Z_n}Y_j,
	\]
	with \(Y_j\) i.i.d.\ copies of \(Y\) and $Z_0 = 1$. The process $(Z_n)_{n \geq 0}$ is a Galton-Watson process. The mean of the offspring distribution \(m=\E[Y]\) determines extinction: if \(m\le1\), extinction occurs almost surely (\(q=1\)); if \(m>1\), the extinction probability is less than one, i.e. \(q<1\). The probability $q$ is the minimal solution of the equation
	\[
	q=\sum_{k\ge0}p_kq^k.
	\]
	Every realization of the process corresponds to a rooted plane tree that encodes its parent–child relations. For further details about these processes see, for instance, \cite{AthreyaNey} and \cite{Neveu}.

	\section{Lagrange's inversion formula}
	
	Let $\psi$ be a power series in $\K$ with radius of convergence $R_\psi > 0$. Let $g$ be the  power series which is the (unique) solution of Lagrange's equation with data $\psi$, i.e. satisfying:
	$$(\star) \quad \quad g(z)=z \psi(g(z)).$$
	
	The following classical result describes an explicit relationship between the coefficients of the power series \( \psi \) and the solution \( g \), thereby providing an explicit power series representation of \( g \) as a solution to Lagrange's equation \((\star)\).
	
	\begin{thmletra}[Lagrange's inversion formula]
		Let \( \psi \in \mathcal{K} \) be a power series with radius convergence $R_{\psi}>0$, and consider the function \( g(z) \) defined as the solution to Lagrange's equation with data \( \psi \):
		\[
		g(z) = z\psi(g(z)).
		\]
		Then, the coefficients of \( g(z) \) satisfy the relation:
		\[
		A_n = \textsc{coeff}_n[g(z)] = \frac{1}{n} \textsc{coeff}_{n-1}[\psi(z)^n], \quad \text{for any } n \geq 1,
		\]
		with \( A_0 = 0 \).
		\smallskip
		
		More generally, for any holomorphic power series \( H \) such that \( H(g(z)) \) is well-defined, we have:
		\[
		\textsc{coeff}_n[H(g(z))] = \frac{1}{n} \textsc{coeff}_{n-1}[H^{\prime}(z) \psi(z)^n], \quad \text{for any } n \geq 1.
		\]
	\end{thmletra}~See, for instance, \cite{Gessel2016} for a proof of this result. 
	\medskip

	The inverse function of \( g(z) \) is given by \( z/\psi(z) \). This will be relevant later on, particularly in connection with Lagrange's inversion formula. This explicit form of the inverse enables a natural reparametrization of the Khinchin family associated with \( g \). This result, as we will see later, is fundamental in the study of combinatorial and random rooted trees.
	\smallskip
	
	We will denote by \( (Z_t) \) the Khinchin family associated with \( g \), the solution to Lagrange's equation with data \( \psi \). This Khinchin family will play a relevant role later on. Notice that $g$, the solution to Lagrange's equation with data $\psi$, is, by definition, a power series in the class $\K_s$, in fact $g(0) = 0$ and $g^{\prime}(0) = A_1 =b_0 \neq 0$, and therefore $(Z_t)$ is a shifted Khinchin family.

	\subsection{Function $t/\psi(t)$}
	
	Lagrange's equation gives that $g$ and $z/\psi(z)$ are inverse of each other and thus, for power series $\psi \in \K^{\star}$, we have that
	$$g(z/\psi(z))=z\, , \quad \mbox{ for } z \in \Omega_g = g(\D(0,R_g)) \subseteq \D(0,R_{\psi})\,.$$
Invoking Lagrange’s equation together with the fact that $\psi\in\mathcal \K$, one sees that $\psi$ cannot vanish in $\Omega_g$.
	In fact, since $g$ has non-negative coefficients, we conclude that
	\begin{align*}
	g(t/\psi(t))=t\, , \quad \mbox{for $t \in[0, \tau]$}\,,
	\end{align*}
	that is, the image of the interval $[0,\rho]$ by $g$ is the interval $g([0,\tau/\psi(\tau)]) = [0,\tau]$. One may verify this claim directly (using that $g$ is increasing and continuous on the closure of $\D(0,\rho)$), or else deduce it as a corollary of Theorem \ref{thm: Otter-Meir-Moon} below.
	\medskip
	
	For the function $t \mapsto t/\psi(t)$, we have
	\begin{lem}\label{lemma: monotony_t/psi(t)} For  $\psi \in \K^\star$ with apex $\tau \in (0,R_\psi)$, the function $t/\psi(t)$
		\begin{enumerate}
			\item is strictly increasing on the interval $[0,\tau)$
			\item is strictly decreasing on the interval $(\tau,R_{\psi})$
			\item has a maximum at $t = \tau$.
		\end{enumerate}
		For \(\psi\in\mathcal K\setminus\mathcal K^\star\), the function \(t/\psi(t)\) is strictly increasing for all \(t\in(0,R_\psi)\).
	\end{lem}
	
	\begin{proof}
		The result follows from the identity
		\begin{align*}
			\left(\frac{t}{\psi(t)}\right)^{\prime} = \frac{1}{\psi(t)}(1-m_{\psi}(t)), \quad \text{ for all } t \in [0,R_{\psi}).
		\end{align*}
		Recall that for power series $\psi \in \K$ the mean $m_{\psi}(t)$ is an increasing function and for power series $\psi \in \K^{\star}$, there exists $\tau \in (0,R_{\psi})$ such that $m_{\psi}(\tau) = 1$.
	\end{proof}
	
	Fix $\psi \in \K^{\star}$ such that $\psi^{\prime}(0) \neq 0$, from the previous lemma we obtain that the probability \(\P(Y_t=1)\) increases, attains its maximum at \(t=\tau\), and then decreases; whereas if the power series \(\psi\in\mathcal K\setminus\mathcal K^\star\), the probability \(\P(Y_t=1)\) is strictly increasing on \((0,R_\psi)\).
	
	\subsection{Radius of convergence of $g$} Now we study the radius of convergence of the solution to Lagrange's equation with data $\psi \in \K$.
	
	\begin{theo}\label{thm: radius_g} Let $\psi \in \K$ be a power series with radius of convergence $R_{\psi}>0$. Denote $g$ the solution of Lagrange's equation with data $\psi$, that is, $g(z) = z\psi(g(z))$\,,	
		then the radius of convergence of $g$ is given by 
		\[
		R_g = \sup_{t \in (0,R_{\psi})}\frac{t}{\psi(t)}.
		\]\\[.1 cm]
		In case that $\psi \in \K^{\star}$ we have $R_g = \rho = \tau/\psi(\tau)$\,.
	\end{theo}
	
%
%

	The radius of convergence of this power series, as we shall see later, is closely related to the behavior of various probabilities associated with a Galton-Watson process.
	
	\smallskip
	In the case where \( \psi \in \mathcal{K}^{\star} \) and \( \textsc{coeff}_1[\psi(z)] = b_1 = 1 \), the radius of convergence of the solution \( g \) to Lagrange's equation with data \( \psi \) coincides with the probability.
	\begin{align*}
		\P(Y_{\tau} = 1) = \frac{\tau}{\psi(\tau)} = \rho\,. 
	\end{align*}
	More generally: for power series \( \psi \in \mathcal{K}^{\star} \) such that \( b_1 \geq 1 \) we have that \( \rho \leq 1 \).
	\smallskip
	
	Note that when \(\psi\in\mathcal K^\star\), the Khinchin family \((Z_t)\) associated with the solution of Lagrange’s equation is naturally parametrized by \(t\in[0,\rho]\), where \(\rho=\tau/\psi(\tau)\). For further details about extending Khinchin families to boundary values we refer to \cite{K_dos}.
	\smallskip
	
	\subsection{Otter-Meir-Moon asymptotic formula}
	
	We now present an asymptotic formula, for the coefficients \( A_n \) of the solution to Lagrange’s equation with data \( \psi \in \mathcal{K}^\star \), originally due to Otter, Meir, and Moon, see \cite{MeirMoon,Otter}. For additional extensions, we refer the reader to~\cite{K_dos}, where a proof based on Khinchin families is provided.
	\smallskip
	
	Recall that for power series $\psi(z) = \sum_{n \geq 0}b_nz^n \in \K$ we denote 
		\[
	Q_{\psi}\triangleq \gcd\{n \geq 1 : b_n \neq 0\} = \lim_{N \to \infty} \gcd\{1 \leq n \leq N : b_n \neq 0\}.
	\]
	
	
	%
	%

	\begin{thmletra}[Otter-Meir-Moon, \cite{MeirMoon,Otter}]\label{thm: Otter-Meir-Moon} Let \( \psi \in \mathcal{K}^{\star} \) be a power series. Then, for indices $n \equiv 1$, $\mod Q_{\psi}$, the coefficients \( A_n \) of the solution \( g \) to Lagrange’s equation with data \( \psi \) satisfy the asymptotic estimate
		\begin{align*}
			A_n \sim \frac{Q_{\psi}}{\sqrt{2\pi}\, \sigma_{\psi}(\tau)}  \frac{\psi(\tau)^n}{\tau^{n-1}}  \frac{1}{n^{3/2}}, \quad \text{as } n \to \infty 
		\end{align*}
		Here, \( \tau \) denotes the apex of \( \psi \) (i.e $m_{\psi}(\tau) = 1$) 
	\end{thmletra}
	
	For any $\psi \in \K^{\star}$, the previous asymptotic formula for the coefficients $A_n$ of the solution to Lagrange's equation $g$ with data $\psi$ gives that, for indices $n \equiv 1$, $\mod Q_{\psi}$, we have
	\begin{align*}
		A_n\rho^{n} \sim \frac{Q_{\psi} \, \tau}{\sqrt{2\pi}\sigma_{\psi}(\tau)}\frac{1}{n^{3/2}}\,, \quad \text{ as } n \rightarrow \infty\,.
	\end{align*}
	This implies that for any \(\psi\in\mathcal{K}^\star\), the solution \(g\) extends continuously to the closure of the disk \({\D(0,\rho)}\).  In particular,  
	\[
	\lim_{t\uparrow\rho}g(t)=\tau,
	\]  
	so it is natural to set \(g(\rho)=\tau\). We will use this relation repeteadly along this note. 
	
	\section{Lagrange's equation and simple varieties of trees}
	
	Fix $\psi(z) = \sum_{n \geq 0}b_nz^n \in \K$ with radius of convergence $R_{\psi}>0$. A simple variety of trees is a family of weighted rooted plane trees whose weighted generating function is the solution of Lagrange's equation with data $\psi$, that is, is a power series $g$ verifying the relation
	\begin{align*}
		g(z) = z\psi(g(z))\,.
	\end{align*}
	Alternatively some authors define these simple varieties for power series $\psi \in \K^{\star}$, this is not the case here. For further details see, for instance, \cite{Drmota},  \cite[p. 452]{Flajolet} and \cite{Janson2012}. 
	\smallskip
	
	We now relate the simple varieties of trees with the class of rooted plane trees endowed with the weight
	\begin{align}\label{eq:weights}
		\omega(a) = \prod_{j = 0}^{\infty}b_j^{k_j(a)}\,, \quad \text{ for any } a \in \mathcal{G}\,.
	\end{align}
	Here $k_j(a)$ counts the number of nodes in $a$ with exactly $j$ descendants. Observe that the weight $\omega$ depends on the power series $\psi$, but we omit this from the notation for simplicity. The next result is standard, see \cite{Pitman} for similar arguments.
	\begin{propoletra}\label{propo: weights} Let $\psi \in \K$ be a power series with positive radius of convergence, then 
		\begin{align*}
			\textsc{coeff}_n[g(z)] = A_n = \sum_{a \in \mathcal{G}_n}\omega(a)
		\end{align*}
		here $g$ denotes the solution to Lagrange's equation with data $\psi$.
	\end{propoletra}
	
	The proof of this result follows by combining Lagrange’s inversion formula with the multinomial theorem (and counting rooted plane trees having the same list of outdegrees $(k_0,k_1,\dots)$); see \cite{Pitman} for similar arguments.
	\smallskip
	
		%
		%
	In a similar way, for a forest with \(m\) rooted trees and \(n\) nodes, the $n$-th coefficient of \(g(z)^m\) equals the sum of the weights of all rooted forests with \(n\) nodes and \(m\) rooted trees.
	
	\section{The parametric Galton Watson process}
	
	Given a power series $\psi \in \K$ and Khinchin family $(Y_t)$, in this section we study the Galton-Watson process with offspring distribution $Y_t$. This construction gives rise to a family of parametric Galton-Watson processes $T_t$ which depends on $t \in (0,R_{\psi})$. For each $t \in (0,R_{\psi})$ the rooted random tree $T_t$ takes values on the Polish space $\mathbb{T} = \mathcal{G} \cup \mathcal{G}_{\infty}$.
	
	\subsection{The Galton-Watson process with offspring distribution $Y_t$}
	
	Fix \( \psi \in \mathcal{K} \) a power series with radius of convergence \( R_{\psi} > 0 \), and denote by \( (Y_t)_{t \in (0,R_{\psi})} \) the Khinchin family associated with \( \psi \). Throughout this section, we consider these as fixed, unless explicitly stated otherwise.
	\smallskip
	
	For each \( t \in (0,R_{\psi})\), we can generate a Galton-Watson process \( T_t \) with offspring distribution \( Y_t \). 
	We denote by \( (T_t)_{t \in (0,R_{\psi})} \) the parametric family of Galton-Watson processes, where for each \( t \in (0,R_{\psi}) \), 
	the random tree \( T_t \) has an offspring distribution given by \( Y_t \). The random variable \( T_t \) takes values in \( \mathbb{T} \) and has a probability mass function given by:
	\begin{align}\label{eq: measure}
		\P(T_t = a) = \prod_{j = 0}^{\infty}\P(Y_t = j)^{k_j(a)}\,, \quad \text{ for any } a \in \mathbb{T}\,,
	\end{align}
	with the convention that $0^0 = 1$. Recall that $k_j(a)$ counts the number of nodes in $a$ with exactly $j$ descendants. This measure assigns probability $0$ to the singletons in $\mathcal{G}_{\infty}$. 
	\smallskip
	
	Before moving forward, we prove the following lemma. This result establishes a connection between the probabilistic and combinatorial interpretations of these weighted rooted trees.
	\begin{lem} \label{lemma: weight_prob}
		Fix \( t \in (0,R_{\psi}) \). Then, for any \( n \geq 1 \) and \( a \in \mathcal{G}_n \), we have
		\begin{align*}
			\P(T_t = a) = \frac{\omega(a)t^{n-1}}{\psi(t)^n}\,,  
		\end{align*}
		with the convention that $0^0 = 1$. Here, \( \omega(a) \) denotes the weight of the rooted plane tree \( a \in \mathcal{G}_n \).
	\end{lem}
	
	\begin{proof} 
		For any \( n \geq 1 \) and \( a \in \mathcal{G}_n \), we have
		\begin{align*}
			\P(T_t = a) &= \prod_{j = 0}^{n-1} \P(Y_t = j)^{k_j(a)} 
			= \prod_{j = 0}^{n-1} \left( \frac{b_j t^j}{\psi(t)} \right)^{k_j(a)}.
		\end{align*}
		Recalling that for a tree \( a \in \mathcal{G}_n \), the following relations hold:
		\begin{align*}
			k_0(a) + k_1(a) + k_2(a) + \dots + k_{n-1}(a) &= n\,, \\
			k_1(a) + 2k_2(a) + \dots + (n-1)k_{n-1}(a) &= n-1\,,
		\end{align*}
		these equations express that the tree has \( n \) nodes and \( n-1 \) edges (or, equivalently, $n-1$ nodes that descend from a parent node, we don't count the root), we obtain that
		\begin{align*}
			\P(T_t = a) &= \prod_{j = 0}^{n-1} b_j^{k_j(a)} \frac{t^{n-1}}{\psi(t)^n} 
			= \frac{\omega(a)t^{n-1}}{\psi(t)^n}.
		\end{align*}
	\end{proof}
	
	\subsection{Probability generating functions} In this section, we study and prove various relations among the probability generating functions of certain variables associated with \( T_t \) and the Khinchin family of $g$, the solution to Lagrange's equation with data \( \psi \).
	\smallskip

	We denote by $|T_t|$ the total progeny of the random tree $T_t$, that is, the number of nodes in the random tree $T_t$. We have the following proposition. 
	\begin{propo}\label{propo: total_progeny}For each $t \in (0,R_{\psi})$ we have 
		\begin{align*}
			\P(|T_t| = n) = \P(T_t \in \mathcal{G}_n) = \frac{A_nt^{n-1}}{\psi(t)^n}\,, \quad \text{ for any } n \geq 1\,.
		\end{align*}
		Moreover, for \(\psi\in\mathcal K^\star\), we have \(Z_{t/\psi(t)}\stackrel{d}{=}|T_t|\), for all \(t\in[0,\tau]\) and for $\psi \in \K \setminus \K^{\star}$ we have \(Z_{t/\psi(t)}\stackrel{d}{=}|T_t|\), for all $t \in (0,R_{\psi})$.
	\end{propo}
	\begin{proof} We have that
		\begin{align*}
			\P(T_t \in \mathcal{G}_n) = \sum_{a \in \mathcal{G}_n}\P(T_t = a) = \sum_{a \in \mathcal{G}_n}\frac{\omega(a)t^{n-1}}{\psi(t)^n} = \frac{A_n t^{n-1}}{\psi(t)^n}
		\end{align*}
		In the last two equalities we use Lemma \ref{lemma: weight_prob} and Proposition \ref{propo: weights}. Recall that $A_n$ denotes the $n$-th coefficient of the solution to Lagrange's equation with data $\psi$.
		\smallskip
		
		To prove the second part of the proposition observe that for any $t \in [0,\tau]$ we have
		\begin{align*}
			\P(Z_{t/\psi(t)} = n) = \frac{A_n (t/\psi(t))^n}{g(t/\psi(t))} = \frac{A_n t^{n-1}}{\psi(t)^n}\,, \quad \text{ for any } n \geq 1\,.
		\end{align*}
		In the last equality we use that for $t \in [0,\tau]$, the function $t/\psi(t)$ is the inverse function of $g$, that is, we have the equality $g(t/\psi(t)) = t$, for any $t \in [0,\tau]$. The same holds if $\psi \in \K \setminus \K^{\star}$, but in this case, for every $t \in [0, R_{\psi})$.
	\end{proof}
	
	For each $t \in (0,R_{\psi})$, we let $\psi_t(z) = \psi(tz)/\psi(t)$ be the probability generating function of the random variable $Y_t$; we set $\psi_0\equiv1$ which is consistent with $Y_0 \equiv 0$.
	\smallskip
	
	We consider, for each $t \in [0,R_{\psi})$, Lagrange's equation with data $\psi_t$ and the corresponding  power series solution $g_t(z)$:
	\begin{align}\label{eq: Lagrange-probabilitygen}
		g_t(z) = z\psi_t(g_t(z)).
	\end{align}
	
	The case $t = 0$ is a degenerate case: $\psi_0(z) \equiv 1$ and $g_0(z) = z$. 
	
	\begin{propo}\label{lemma: g_t-probgen} For $t \in (0,R_\psi)$ we have that
		\begin{align}\label{eq:identity for gt}
			g_t(z) = \frac{g(tz/\psi(t))}{t}, \quad \quad  \mbox{for all $|z| \leq 1$}\,.
		\end{align}	
		The holomorphic function \( g_t \) is continuous on \( \partial \mathbb{D} \) and satisfies \( g_t(\mathbb{D}) \subset \mathbb{D} \). Here, \( g \) denotes the solution to Lagrange's equation with data \( \psi \).
		\smallskip
		
		The power series \(g_t(z)\) is the probability generating function of \(Z_{t/\psi(t)}\) whenever either \(\psi\in\mathcal{K}^\star\) and \(t\in[0,\tau]\), or \(\psi\in\mathcal{K}\setminus\mathcal{K}^\star\) and \(t\in(0,R_\psi)\).
	\end{propo}
	\begin{proof} Solutions to Lagrange's equation with a given data are unique. Let us check that both power series satisfy the same Lagrange's equation with data \( \psi_t \), and therefore, they must be the same power series.
		\smallskip
		
		The power series $g_t(z)$ verifies Lagrange's equation with data $\psi_t$ by definition. For any $t \in (0,R_{\psi})$, denote $f_t(z) =  {g(tz/\psi(t))}/{t}$. Notice that 
		\begin{align*}
			z\psi_t(f_t(z)) = z\frac{\psi(g(tz/\psi(t)))}{\psi(t)} = \frac{g(tz/\psi(t))}{t} = f_t(z)\,.
		\end{align*}
		Here we use that $g$ verifies Lagrange's equation with data $\psi$. Therefore, by the uniqueness of solutions, it follows that
		\begin{align*}
			g_t(z) = \frac{g(tz/\psi(t))}{t}\,, \quad \text{ for all } t \in (0,R_{\psi}) \text{ and } |z|\leq 1\,.
		\end{align*}
	\end{proof}

	Fix \( \psi \in \mathcal{K}^{\star} \). Observe that, for any \( t \in (0, R_{\psi}) \), the radius of convergence of the power series \( g_t(z) \) is given by
	\[
	R_{g_t} = \frac{\tau / \psi(\tau)}{t / \psi(t)} \geq 1,
	\]
	and furthermore, \( g_t(R_{g_t}) = \tau / t \).

	\subsection{Extinction probability}\label{subsec: extinction_tree}
	
	We now analyze the extinction probability function for the parametric Galton-Watson process \( T_t \) and derive explicit formulas for this probability in terms of the coefficients of the solution \( g \) to Lagrange's equation with data \( \psi \).
	\smallskip
	
	The remarkable phenomenon occurring here, well known in the literature, is the transition from a discrete probability measure to a mixture of two measures as the parameter \( t \) varies. Under the present setting, this behavior is reminiscent of a phase transition in statistical mechanics.
	\smallskip
	
	
	For each $t \in (0,R_{\psi})$ denote $q(t)$ the probability of extinction of the parametric Galton-Watson process $T_t$. This probability is given by
	\begin{align*}
		q(t) = \P(T_t \in \mathcal{G})\,, \quad \text{ for any } t \in (0,R_{\psi})\,.
	\end{align*}
	
	We now provide an explicit expression for this probability in terms of the coefficients $A_n$ of the solution to Lagrange’s equation with data \( \psi \).
	\begin{theo}\label{thm:extinction_tree}
		The extinction probability of the Bienaymé-Galton-Watson process with offspring distribution \( Y_t \) satisfies the following explicit formula:
		\begin{equation}
			q(t) = \sum_{n=1}^{\infty} \frac{A_n t^{n-1}}{\psi(t)^n} = \frac{g(t/\psi(t))}{t}, \quad \text{ for any } t \in (0,R_{\psi})\,,
		\end{equation}
		where \( A_n \) are the coefficients of the solution to Lagrange’s equation with data \( \psi \).
	\end{theo}
	\begin{proof} Notice that, for any $t \in (0,R_{\psi})$, we have 
		\begin{align*}
			\P(T_t \in \mathcal{G}) = \sum_{n = 1}^{\infty}\P(|T_t| = n) = \sum_{n = 1}^{\infty}\frac{A_nt^{n-1}}{\psi(t)^n}\,, \quad \text{ for any } n \geq 1\,.
		\end{align*}
		Here we apply Proposition~\ref{propo: total_progeny}, together with the fact that the family \( \{\mathcal{G}_n\}_{n \geq 1} \) forms a countable partition of \( \mathcal{G} \).
	\end{proof}
	
	This formula provides an explicit characterization of extinction probabilities in terms of the coefficients of the solution to Lagrange's equation with data $\psi$.
	\smallskip
	
	An alternative approach to the previous theorem is the following: observe that, for $\psi_t(z) = \psi(tz)/\psi(t)$, we have Lagrange's equation
	\begin{align*}
		g_t(z) = z\psi_t(g_t(z))
	\end{align*}
	and therefore $g_t(1) = \psi_t(g_t(1))$. Because the probability of extinction is the smallest fixed point of the probability generating function $\psi_t(z)$, then $q(t) = g_t(1)$. Observe that since $\psi_t(z)$ is a probability generating function, it satisfies $\psi_t(1)=1$.
	\smallskip
	
	In the previous theorem, the most interesting behavior occurs when \(\psi\in\mathcal K^\star\), since in that case a \emph{phase transition} takes place at \(t=\tau\).
	\smallskip
	
	We now study the properties of the extinction probability \( q(t) \) as a function of $t \in [0,R_{\psi})$. Our first result states that $q(t)$ is a continuous function. 
	\begin{propo} The probability of extinction \( q(t) \) is continuous on the interval \( [0, R_{\psi}) \).
	\end{propo}
	\begin{proof} If $\psi \in \K \setminus \K^{\star}$, we have $q(t) \equiv 1$, for any $t \in [0,R_{\psi})$, which is continuous. Assume that \( \psi \in \K^{\star} \), the function \( t/\psi(t) \) is continuous for any \( t \in [0, R_{\psi}) \). Now, combining Theorem \ref{thm: radius_g} with Lemma \ref{lemma: monotony_t/psi(t)}, we find that
		\[
		\frac{t}{\psi(t)} \leq \frac{\tau}{\psi(\tau)} = \rho = R_g, \quad \text{for all } t \in (0, R_{\psi}),
		\]
		and therefore, since the solution \( g \) to Lagrange's equation is continuous on the closure of the disk \( \mathbb{D}(0, \rho) \), we conclude that
		\[
		q(t) = {g(t/\psi(t))}/{t}
		\]
		is continuous for all \( t \in (0, R_{\psi}) \).
		\smallskip
		
		For \( t = 0 \), observe that
		\[
		g(t/\psi(t)) = A_1 \frac{t}{\psi(t)} + o(t) = \psi(0) \frac{t}{\psi(t)} + o(t), \quad \text{as } t \downarrow 0,
		\]
		and therefore we conclude that \( q(0) = 1 \), as expected. Here we apply Lagrange’s inversion formula to obtain \( A_1 = \textsc{coeff}_0[\psi(z)] =  \psi(0)\).
	\end{proof}
	\smallskip
	
	We now analyze the asymptotic behaviour of $q(t)$, as $t \uparrow R_{\psi}$.
	
	\begin{propo} If $\lim_{t \uparrow R_{\psi}}\psi(t)/t = +\infty$, then 
		\begin{align*}
			q(t) \sim \frac{\psi(0)}{\psi(t)} = \P(Y_t = 0)\,, \quad \text{ as } t \uparrow R_{\psi}\,.
		\end{align*}
	\end{propo}
	\begin{proof} Theorem~\ref{thm:extinction_tree} gives that
		\begin{align*}
			q(t) = \frac{g(t/\psi(t))}{t}, \quad \text{for any } t \in (0, R_{\psi}),
		\end{align*}
		therefore, using the fact that \( g(z) \) is continuous on the closure of \( \mathbb{D}(0, \rho) \), together with the limit \( \lim_{t \uparrow R_{\psi}} t/\psi(t) = 0 \), we find
		\begin{align}\label{eq: tq_cero}
			\lim_{t \uparrow R_{\psi}} t q(t) = \lim_{t \uparrow R_{\psi}} g({t}/{\psi(t)}) = g(0) = 0.
		\end{align}
		
		Now, using that \( q(t) \) is a fixed point of \( \psi_t(z) = \psi(tz)/\psi(t) \), we obtain
		\begin{align*}
			\psi(t q(t)) = q(t) \psi(t),
		\end{align*}
		and hence, using equation \eqref{eq: tq_cero}, we conclude that
		\begin{align*}
			q(t) \sim \frac{\psi(0)}{\psi(t)}, \quad \text{as } t \uparrow R_{\psi}.
		\end{align*}
			Observe that $\psi(0)/\psi(t) = \P(Y_t = 0)$.
	\end{proof}
	
	\begin{remark}
		The cases where the ratio \(\psi(t)/t\) remains bounded, as $t \uparrow R_{\psi}$, are the following:
		\begin{itemize}
			\item If \(R_{\psi} = \infty\), then \(\psi(t)\) is a polynomial of degree 1. In this case \(M_{\psi} = 1\) and \(q(t)\equiv1\) for all \(t\geq 0\).
			\smallskip
			\item If \(R_{\psi}<\infty\), then \(\lim_{t\uparrow R_{\psi}}\psi(t)<\infty\) and the probability of extinction \(q(t)\) can be extended to $t = R_{\psi}$.
	        \smallskip
%
%
%
		\end{itemize}\finremark
	\end{remark}
	
	Now we give an upper bound for the extinction probability. 
	\begin{lem}\label{lemma: bound_q} Fix $\psi \in \K^{\star}$. For any \(t\in(0,R_\psi)\), the extinction probability satisfies
		\[
		q(t)\;\le\;{\tau}/{t}.
		\]
	\end{lem}
	\begin{proof}
		Recall that $g_t(R_{g_t}) = \tau/t$. Since $g_t(z)$ has non negative Taylor coefficients, and $R_{g_t} \geq 1$ for any $t \in (0,R_{\psi})$, we find that 
		$$q(t) = g_t(1) \leq g_t(R_{g_t}) = \tau/t, \quad \text{ for any } t \in (0,R_{\psi})\,,$$ 
		which yields the desired upper bound on the extinction probability. Note that this estimate is meaningful only when $t > \tau$, otherwise $q(t) \equiv 1$. 
	\end{proof}

	The function $q(t)$ is differentiable for all $t\neq\tau$. The following proposition precisely describes the discontinuity in its derivative at that point. 
	\begin{theo}
		Fix \(\psi\in\mathcal K^\star\). The extinction probability \(q(t)\) is differentiable on the interval \((0,R_\psi)\setminus\{\tau\}\), with a jump discontinuity of its derivative at \(t=\tau\). Concretely,
		\[
		\lim_{t\downarrow\tau}\frac{q(t)-1}{t-\tau}=-\frac{2}{\tau},
		\qquad
		\lim_{t\uparrow\tau}\frac{q(t)-1}{t-\tau}=0.
		\]
		Moreover, for all \(t\in(0,R_\psi)\setminus \{\tau\}\), one has the closed‐form for the derivative
		\[
		q'(t)
		=\frac{q(t)}{t}\left(\frac{1-m_\psi(t)}{1-m_\psi(tq(t))}-1\right).
		\]
	\end{theo}
	\begin{proof} Observe that \(t = \tau\) is the unique point at which \(q(t)\) may fail to be differentiable. By Lemma~\ref{lemma: monotony_t/psi(t)}, the function \(t/\psi(t)\) is strictly decreasing on \((\tau,R_{\psi})\). Hence for all \(t>\tau\),
		\[
		\frac{t}{\psi(t)} < \frac{\tau}{\psi(\tau)},
		\]
		and since \(g\) is analytic (and thus differentiable) on \(\mathbb{D}(0,\tau/\psi(\tau))\), it follows that 
		\[
		q(t) = {g(t/\psi(t))}/{t}
		\]
		is differentiable for every \(t>\tau\). For $t < \tau$ we have $q(t) \equiv 1$, then $q(t)$ is differentiable. Let's prove that $q(t)$ is not differentiable at $t = \tau$. For the left-limit we have
		\begin{align*}
			\lim_{t \uparrow \tau} \frac{q(t) - 1}{t - \tau} = 0.
		\end{align*}
		This conclusion follows directly from the fact that \(q(t)\equiv1\) for all \(t\in[0,\tau]\). Now we compute the derivative of $q(t)$, for $t \in (\tau,R_{\psi})$. Using that $g(z) = z\psi(g(z))$ we find  
		\begin{align}\label{eq: derivative_lagrange}
			g^{\prime}(z) = \frac{\psi(g(z))}{1-z\psi^{\prime}(g(z))}\,, \quad \text{ for any } |z|<\rho\,.
		\end{align}
		For any \(t \in (0,R_{\psi}) \setminus \{\tau\}\), combining Theorem \ref{thm:extinction_tree}, equation \eqref{eq: derivative_lagrange} and the identity \(\psi(tq(t))=q(t)\psi(t)\) we obtain that
		\begin{align}\label{eq: derivada_q}
			q'(t)=\frac{q(t)}{t}{\left(\frac{1-m_{\psi}(t)}{1-m_{\psi}(tq(t))}-1\right)},
		\end{align}\\[.001 cm]
		here we use that \(q(t)\) is differentiable on \((0,R_{\psi}) \setminus \{\tau\}\). 
		\smallskip
		
		Now we compute the right derivative of $q(t)$,  at $t = \tau$. Taylor’s expansion around \(t=\tau\) gives
		\[
		1-m_{\psi}(t) = -m_{\psi}'(\tau)\,(t-\tau)\,(1+o(1))\,, 
		\quad \text{ as } \quad t\downarrow\tau\,.
		\]
	  Notice that \(m_{\psi}'(\tau)>0\), the mean $m_{\psi}$ is strictly increasing. From the previous asymptotic relation it follows that
		\begin{align}
			\label{eq: medias_q_cociente}
			\frac{1-m_{\psi}(t)}{1-m_{\psi}(tq(t))}
			\sim -\frac{(t-\tau)}{(\tau-tq(t))}\,, 
			\quad \text{ as } \quad t\downarrow\tau\,.
		\end{align}
		
		Denote $h(t) = \psi(t)/t$, for $t \in (0,R_{\psi})$. This function is invariant under $tq(t)$ meaning that $h(tq(t)) = h(t)$. Taylor's expansion around $t = \tau$ gives that 
		\begin{align*}
			h(t) = h(\tau)+(h^{\prime\prime}(\tau)/2)(t-\tau)^2 +O((t-\tau)^3)\,, \quad \text{ as } t \rightarrow \tau\,.
		\end{align*}
		From the previous expansion we obtain that
		\begin{align*}
			\lim_{t \downarrow \tau}\frac{h(tq(t))-h(\tau)}{(tq(t)-\tau)^2} = \frac{h^{\prime\prime}(\tau)}{2} = \frac{\sigma_{\psi}^2(\tau)}{2}\frac{\psi(\tau)}{\tau^3}>0
		\end{align*}
		and
		\begin{align*}
			\lim_{t \downarrow \tau}\frac{h(t)-h(\tau)}{(t-\tau)^2} = \frac{h^{\prime\prime}(\tau)}{2} = \frac{\sigma_{\psi}^2(\tau)}{2}\frac{\psi(\tau)}{\tau^3}>0\,.
		\end{align*}
		\smallskip
		
		Using that $h(tq(t)) = h(t)$ we find  
		\begin{align*}
			\lim_{t \downarrow \tau}\frac{(tq(t)-\tau)^2}{(t-\tau)^2} = 1
		\end{align*}
		and this implies that 
		\begin{align}\label{eq: limit_tq}
			\lim_{t \downarrow \tau}\frac{tq(t)-\tau}{t-\tau} = -1\,.
		\end{align}
		Recall that $tq(t) \leq \tau$, for all $t \in (0, R_\psi)$; see Lemma~\ref{lemma: bound_q}. Hence, combining equations \eqref{eq: derivada_q} and \eqref{eq: medias_q_cociente} with equation \eqref{eq: limit_tq}, we conclude that 
		\[
		\lim_{t\downarrow\tau}q'(t)
		=\lim_{t\downarrow\tau}\frac{q(t)-1}{t-\tau}
		=-\frac{2}{\tau}\,.
		\]
	\end{proof}
	
	\begin{remark}
		From the expression
		\[
		q'(t)
		=\frac{q(t)}{t}\!\left(\frac{1-m_{\psi}(t)}{1-m_{\psi}(tq(t))}-1\right), \quad \text{ for any } t \in (\tau,R_{\psi}),
		\]
		we deduce that \(q'(t)<0\) for all \(t\in(\tau,R_{\psi})\), and therefore $q(t)$ is strictly decreasing on that interval. 
		\finremark
	\end{remark}
	
	\smallskip
	
	\subsection{The \textit{uniform} probability}
	
	We have previously discussed the weight associated to the family of rooted plane trees. This weight  is defined as follows: for any \( a \in \mathcal{G} \), we have
	\begin{align*}
		\omega(a) = \prod_{j = 0}^{\infty}b_j^{k_j(a)}\,, 
	\end{align*}
	with the convention that \( 0^0 = 1 \). By means of $\omega$, see Proposition \ref{propo: weights} above, we can build all the families of simple varieties of trees.  
	\smallskip
	
	Now, for any $a \in \mathcal{G}$, consider the probabilistic weight  
	\begin{align*}
		\omega_{\P}(a) = \prod_{j=0}^{\infty}\P(Y_t = j)^{k_j(a)} = \frac{\omega(a)t^{|a|-1}}{\psi(t)^{|a|}} = \P(T_t = a)\,,
	\end{align*}
	again with the convention that \( 0^0 = 1 \). Assume, for simplicity, that \( A_n \neq 0 \) for any \( n \geq 1 \), otherwise, we can restrict \( n \geq 1 \) to values satisfying \( n \equiv 1 \mod Q_{\psi} \). First, we consider the measure on \( \mathcal{G} \) with weights \( \omega_{\P} \) and \( \omega \). For a rooted plane tree $a \in \mathcal{G}$ we compute its measure as follows
	\begin{align}\label{eq: Boltzmann/condition_extinction}
		\mu_{t}(\{a\}) = \frac{\omega_{\P}(a)}{\sum_{a \in \mathcal{G}}\omega_{\P}(a)} = \frac{\omega(a)t^{|a|-1}/\psi(t)^{|a|}}{\sum_{n \geq 1}A_n t^{n-1}/\psi(t)^n} = \frac{\omega(a)t^{|a|-1}/\psi(t)^{|a|}}{q(t)} = \frac{\P(T_t = a)}{q(t)}
	\end{align}
	We will show below that this measure coincides with conditioning the Galton–Watson tree on the extinction event. 
	\smallskip

	Now we compute the measure $\mu$ with respect to \( \omega \) and \( \omega_{\P} \), respectively, for the set of rooted plane trees with \( n \) nodes. For any $a \in \mathcal{G}_n$, we have
	\begin{align*}
		\mu_{\omega,n}(\{a\}) &=	\frac{\omega(a)}{\sum_{a \in \mathcal{G}_n}\omega(a)} =  \frac{\omega(a)}{A_n}\,, \\ 
		\mu_{\omega_{\P},n}(\{a\}) &=	\frac{\omega_{\P}(a)}{\sum_{a \in \mathcal{G}_n}\omega_{\P}(a)} = \frac{\P(T_t = a)}{A_n t^{n-1}/\psi(t)^n} 
		= \frac{\omega(a)t^{n-1}/\psi(t)^n}{A_n t^{n-1}/\psi(t)^n} = \frac{\omega(a)}{A_n}.
	\end{align*}
	Thus, the measures on \( \mathcal{G}_n \) with either the probabilistic or the combinatorial weight are exactly the same measure. This reflects the fact that conditioning the random tree \( T_t \) on the event \( |T_t| = n \) yields the \textit{uniform} probability distribution over subclasses of weighted trees. Here, "uniform" is understood in a broad sense, which we will clarify in the following.
	\smallskip
	
	Fix \( \mathcal{R} \subseteq \mathcal{G} \) as a subclass of rooted plane trees, and for any \( n \geq 1 \), define  
	\begin{align*}
		R_n = \sum_{a \in \mathcal{R} \cap \mathcal{G}_n} \omega(a)\,.
	\end{align*}
	We denote by $\mathcal{R}_n = \mathcal{R} \cap \mathcal{G}_n$. Regarding this coefficient (the previous sum of weights), we have the following result:
	\begin{theo}
		For any \( t \in (0,R_{\psi}) \), we have  
		\begin{align*}
			\P(T_t \in \mathcal{R} \mid |T_t| = n) = \frac{R_n}{A_n}, \quad \text{ for any } n \equiv 1 \mod Q_{\psi}\,,
		\end{align*}
		and this probability is independent of the parameter \( t \).
	\end{theo}
	\begin{proof} Fix $t \in (0,R_{\psi})$. We have
		\begin{align*}
			\P(T_t \in \mathcal{R} \mid |T_t| = n) = \frac{\P(T_t \in \mathcal{R} \text{ and } |T_t| = n)}{\P(|T_t| = n)} = \frac{\P(T_t \in \mathcal{R}_n)}{\P(T_t \in \mathcal{G}_n)} 
		\end{align*}
		\smallskip
		
		For any $n \equiv 1, \mod Q_{\psi}$ we have 
		\begin{align*}
			\P(T_t \in \mathcal{R}_n) = \sum_{a \in \mathcal{R}_n}\frac{\omega(a)t^{n-1}}{\psi(t)^n} = \frac{R_n t^{n-1}}{\psi(t)^{n}}
		\end{align*}
		finally Proposition \ref{propo: total_progeny} gives that 
		\begin{align*}
			\P(T_t \in \mathcal{G}_n) = \frac{A_n t^{n-1}}{\psi(t)^n}
		\end{align*}
		Hence, the result follows by taking the ratio of the two probabilities.
	\end{proof}
	
	If we condition upon extinction we get the following result. 
	\begin{theo} For any $t \in (0,R_{\psi})$ we have
		\begin{align*}
			\P(T_t \in \mathcal{R} | \text{ extinction }) = \frac{\P(T_t \in \mathcal{R} \cap \mathcal{G})}{q(t)} = \frac{1}{q(t)}\sum_{n = 1}^{\infty}\frac{R_n t^{n-1}}{\psi(t)^n}\,.
		\end{align*}
		Here $\mathcal{R} \subseteq \mathbb{T}$.
	\end{theo}
	This coincides with the measure defined earlier on the set of finite rooted plane trees \(\mathcal{G}\); see equation~\eqref{eq: Boltzmann/condition_extinction}.
	\smallskip
	
	Extensions of this approach to Galton–Watson forests (the parametric Galton-Watson forest) follow analogously by considering products of power series and using standard independence arguments. Similar coefficient‑based extinction formulas follow immediately. For brevity, we do not develop these results in here.

	\section{Conclusion and further directions}
	The ideas developed in this note extend naturally to subclasses of rooted plane trees defined by various types of restrictions (and also to random forests). Moreover, when combined with the results in~\cite{K_dos}, this perspective yields precise asymptotic formulas across a broad range of settings. A key strength of the approach is that it unifies and simplifies classical arguments, often reducing them to algebraic manipulations, while making the connection between combinatorial and probabilistic viewpoints explicit. The alternative proof of the Otter–Meir–Moon Theorem in~\cite{K_dos}, based on Khinchin families and including several extensions, serves as a clear illustration of this principle. The same methodology can be applied to derive asymptotic formulas for the probabilities associated to Galton-Watson processes studied here, and also to simulate their probabilities efficiently using power series coefficients.
	\smallskip
	
	Finally, we believe that the underlying structure revealed by these results is broadly applicable, suggesting that this methodology may be extended to other weighted combinatorial classes through the framework of Khinchin families.
	For further details about the Theory of Khinchin families see \cite{CFFM1,CFFM2, CFFM3, K_dos, MaciaThesis, MaciaGaussian}.

\end{document}